\documentclass{my-aims}

\usepackage{amsmath}
\usepackage{paralist}
\usepackage[colorlinks=true]{hyperref}

\hypersetup{urlcolor=blue, citecolor=red}

\def\ppages{X--XX}

\newtheorem{theorem}{Theorem}[section]
\newtheorem{lemma}[theorem]{Lemma}
\newtheorem{proposition}{Proposition}
\newtheorem*{problem}{Problem}
\theoremstyle{definition}
\newtheorem{definition}[theorem]{Definition}
\newtheorem{remark}{Remark}

\newtheorem{example}{Example}

\title[Regional enlarged observability of Caputo]%
{Regional enlarged observability of Caputo fractional differential equations}

\author[H. Zouiten, A. Boutoulout and D. F. M. Torres]{}

\subjclass{Primary: 35R11, 93B07; Secondary: 26A33, 93C20.}

\keywords{Fractional evolution systems, Caputo time derivatives, 
Enlarged observability, Regional reconstruction, HUM approach.}

\email{zouiten.hayat1991@gmail.com}
\email{boutouloutali@yahoo.fr}
\email{delfim@ua.pt}

\thanks{This research is part of first author's Ph.D. project, 
which is carried out at Moulay Ismail University, Meknes.}

\thanks{$^*$ Corresponding author: delfim@ua.pt}

\begin{document}

\maketitle

\centerline{\scshape Hayat Zouiten and Ali Boutoulout}
\medskip
{\footnotesize
\centerline{TSI Team, MACS Laboratory, Department of Mathematics and Computer Science,}
\centerline{Moulay Ismail University, Faculty of Sciences,}
\centerline{11201 Meknes, Morocco}
} 

\medskip

\centerline{\scshape Delfim F. M. Torres$^*$}
\medskip
{\footnotesize
\centerline{Center for Research \& Development in Mathematics and Applications (CIDMA),}
\centerline{Department of Mathematics, University of Aveiro}
\centerline{3810--193 Aveiro, Portugal}
}

\bigskip


\begin{abstract}
We consider the regional enlarged observability problem for fractional 
evolution differential equations involving Caputo derivatives. 
Using the Hilbert Uniqueness Method, we show that it is possible 
to rebuild the initial state between two prescribed functions only 
in an internal subregion of the whole domain. Finally, an example 
is provided to illustrate the theory.
\end{abstract}


\section{Introduction} 

Let $\Omega$ be an open bounded subset of  $\mathbb{R}^{n}$, 
with a regular boundary $\partial \Omega$. For $T > 0$, let us denote 
$Q_{T} = \Omega \times [0,T]$ and $\Sigma_{T} = \partial \Omega \times [0,T]$.  
We consider the following time fractional order diffusion system of order $q \in (0,1)$: 
\begin{equation}
\label{se11}
\begin{cases}
{}^{C}_{0}D_{t}^{q}y(x,t) =  A y(x,t) & \hbox{in}  \quad  Q_{T} \\
y(\xi,t) = 0 &  \hbox{on} \quad \Sigma_{T}\\ 
y(x,0)  = y_{0}(x) &  \hbox{in} \quad \Omega.
\end{cases}
\end{equation}
Without loss of generality, we denote $y(t) : = y(x, t)$.
The measurements are given by the following output function: 
\begin{equation}
\label{se12}
z(t) = C y(t), \quad t \in [0,T] \, ,
\end{equation}
where $^{C}_{0}D_{t}^{q}$ denotes the left-sided Caputo fractional order derivative 
with respect to time $t$, $A$ is a second order linear operator with dense domain 
such that the coefficients do not depend on time $t$ and generates a strongly 
continuous semi-group $(R(t))_{t\geq 0}$ on the Hilbert space $L^{2}(\Omega)$. 
We assume that the initial state $y_{0}\in L^{2}(\Omega)$ is unknown and $C$ 
is called the observation operator, which is a linear operator, possibly unbounded, 
depending on the structure and the number $p\in \mathbb{N}$ of the considered sensors, 
with dense domain $D(C) \subseteq L^{2}(\Omega)$ and range in
the observation space $\mathcal{O} = L^{2}(0,T;\mathbb{R}^{p})$.

The history of fractional calculus goes back to more than 300 years when in 1965 
the derivative of order $q= \displaystyle \frac{1}{2}$ was discussed by Leibniz. 
Since then, several mathematicians contributed to this subject over the years. 
We can mention the works of Liouville, Riemann and Weyl, who made major improvements 
to the theory of fractional calculus. The researches continued with the 
contributions from Fourier, Abel, Leibniz, Gr\"{u}nwald, Letnikov and many others. 
During last decades, the investigation of the theory of differential equations 
of a fractional order was motived by the intensive development of the theory 
of fractional calculus. This area of research has attracted great attention 
of many mathematicians, physicists, and engineers, see, e.g., 
\cite{MR2962045,MR3443073,DS,DD,DSS,HI,MR3571004}. It is 
a very useful and valuable tool in the modeling of many real phenomena. 
Indeed, we can find numerous applications in electrochemistry, electromagnetic, 
fluid dynamics, control theory, viscoelasticity, viscoplasticity, traffic, 
economics, aerodynamics, heat conduction and continuum mechanics, see, e.g.,
\cite{AX,BG,CAMA,MR3316531,DIFR,KST,MA,MR2721980,OU,PO,SKM}.

One of the recent research topics in control theory is studying the concept 
of regional observability for fractional partial differential equations 
\cite{GCK,SP}. The authors in these works discuss the problem of 
regional observability, regional boundary observability and regional gradient 
observability of Riemann--Liouville and Caputo type time fractional diffusion systems, 
where the goal is the possibility to reconstruct the initial (state or gradient state) 
just in an interested subregion of the whole domain $\Omega$. For partial 
differential equations, several works deal with this problem. 
We refer the interested reader to \cite{AM,DR,EL,JAI,ZB} 
and references therein for more details.

In this paper, we investigate the regional enlarged observability, 
also so-called observability with constraints on the state, of fractional 
diffusion systems with Caputo fractional derivative, using the 
Hilbert Uniqueness Method (HUM) of Lions \cite{LI,JL}.

The remainder of this paper is organized as follows. Some basic knowledge 
and preliminary results, which will be used throughout the paper, are given 
in Section~\ref{sec:2}. In Section~\ref{sec:3}, we characterize 
the enlarged observability of the system.  Section~\ref{sec:4} is focused 
on the regional reconstruction of the initial state between two prescribed 
functions only in an internal subregion of the evolution domain. Finally, 
we present an example to demonstrate our main result in Section~\ref{sec:5}. 
We end with Section~\ref{sec:6} of conclusions.


\section{Preliminaries}
\label{sec:2}

In this section, we introduce some notations, definitions, 
and preliminary results, which are used in the rest of the paper.

\begin{definition}[See \cite{KST,PO}] 
\label{definition2.1}
The left-sided and the right-sided fractional integrals of order 
$q > 0$ of a function $y(x,t)$ with respect to $t \in [0,T]$ 
are given as
$$
{}_{0}I_{t}^{q}y(x,t) = \displaystyle \frac{1}{\Gamma(q)} 
\int_{0}^{t} (t-s)^{q-1} y(x,s) ds
$$
and
$$
{}_{t}I_{T}^{q}y(x,t) = \displaystyle \frac{1}{\Gamma(q)} 
\int_{t}^{T} (s-t)^{q-1} y(x,s) ds,
$$
respectively, provided the right-hand sides are pointwise defined on $[0,T]$, 
where $\Gamma (q) $ denotes Euler's Gamma function.
\end{definition}

\begin{definition}[See \cite{KST,PO}]
\label{definition2.2}
The left-sided and right-sided Caputo fractional derivatives 
of order $0< q < 1$, of a function $y(x,t)$ with respect to 
$t \in [0, T]$, are defined as 
$$
^{C}_{0}D_{t}^{q}y(x,t) = \displaystyle \frac{1}{\Gamma(1-q)} 
\int_{0}^{t} (t-s)^{-q} \frac{d}{d s}y(x,s) ds
$$
and 
$$
^{C}_{t}D_{T}^{q}y(x,t) = \displaystyle \frac{-1}{\Gamma(1-q)} 
\int_{t}^{T} (s-t)^{-q} \frac{d}{d s}y(x,s) ds,
$$
respectively.
\end{definition}

\begin{lemma}[See \cite{ZJ}]
\label{lemma2.1}
For $t \in [0,T]$, any $y_{0} \in L^{2}(\Omega)$ and $0 < q < 1$, we say that 
the function $y \in L^{2}(0,T;L^{2}(\Omega))$ is a mild 
solution of system \eqref{se11}, denoted by $y(x,\cdot)$, if it satisfies
$$
y(x,t) =  S_{q}(t)y_{0},
$$	
where
$$
S_{q}(t) =  \int_{0}^{\infty} 
 \xi_{q} (\theta) R(t^{q}\theta) d\theta,
$$
$$
\xi_{q} (\theta) = \displaystyle \frac{1}{q} 
\theta^{-1-\frac{1}{q}}\varpi_{q}(\theta^{-\frac{1}{q}}),
$$
$$
\varpi_{q}(\theta)  = \displaystyle \frac{1}{\pi} 
\sum_{n=1}^{\infty} (-1)^{n-1} \theta^{-nq-1} 
\frac{\Gamma(nq+1)}{n!} \sin(n\pi q), 
\quad \theta \in(0,\infty).
$$
\end{lemma}

\begin{remark}[See \cite{MPG}]
\label{remark2.1}
Let $\xi_{q}$ represent the probability density function defined 
on $(0,\infty)$, satisfying 
$$
\xi_{q}(\theta) \geq 0, \quad \theta \in (0,\infty)  
\quad \hbox{and} \quad \int_{0}^{\infty} \xi_{q}(\theta) d\theta = 1.
$$
Then,
$$
\int_{0}^{\infty} \theta^{\nu} \xi_{q}(\theta) d\theta 
= \frac{\Gamma(1+\nu)}{\Gamma(1+q \nu)}, \quad \nu \geq 0.
$$
\end{remark}

Note that the output function \eqref{se12} can be written as
$$
\begin{array}{rll}
z(t)& = & C S_{q}(t) y_{0}\\
    & = & K_{q} (t) y_{0},
\end{array}
$$
where $K_{q} : L^{2}(\Omega) \longrightarrow \mathcal{O}$ is a linear operator. 
To obtain the adjoint operator of $K_{q}$, we have two cases, depending 
on the notions of admissibility of the observation operator $C$. 
\begin{description}
\item[Case 1] $C$ is bounded (i.e., zone sensors).
Let $C : L^{2}(\Omega) \longrightarrow \mathcal{O}$ and 
$C^{*}$ be its adjoint. We get that the adjoint operator 
of $K_{q}$ is given by
$$ 
\begin{array}{rll}
K_{q}^{*} : \mathcal{O} & \longrightarrow &  L^{2}(\Omega)\\
z^{*} & \longmapsto& \displaystyle \int_{0}^{T} S_{q}^{*}(s)C^{*}z^{*}(s)ds.
\end{array}
$$

\item[Case 2] $C$ is unbounded (i.e., pointwise sensors -- see Definition~\ref{definition2.8}).
In this case, we have $ C : D(C) \subseteq L^{2}(\Omega) \longrightarrow \mathcal{O}$ 
with $C^{*}$ denoting its adjoint. In order to give a sense to \eqref{se12}, 
we make the assumption that $C$ is an admissible observation operator 
in the sense of the following definition.
\end{description}

\begin{definition}
\label{definition2.5}
The operator $C$ of system \eqref{se11}--\eqref{se12} is 
an admissible observation operator if there exists 
a constant $M > 0$ such that
\begin{equation*}
\displaystyle \int_{0}^{T} \left\| CS_{q}(s)y_{0} \right\|^{2} ds 
\leq M \left\| y_{0} \right\|^{2}
\end{equation*}
for any $y_{0} \in D(C)$.
\end{definition}

Note that the admissibility of $C$ guarantees that we can extend the mapping 
$$
y_{0} \longmapsto CS_{q}(t)y_{0} = K_{q}(t)y_{0}
$$ 
to a bounded linear operator from 
$L^{2}(\Omega)$ to $\mathcal{O}$. For more details, see, e.g., \cite{PW,SA,WE}. 
Then the adjoint of the operator $K_{q}$ can be defined as
$$ 
\begin{array}{rll}
K_{q}^{*} :  D(K_{q}^{*}) \subseteq \mathcal{O} 
& \longrightarrow &  L^{2}(\Omega)\\
z^{*} & \longmapsto& 
\displaystyle \int_{0}^{T} S_{q}^{*}(s)C^{*}z^{*}(s)ds.
\end{array}
$$ 
Next we introduce the notion of sensors.

\begin{definition}[See \cite{EL}]
\label{definition2.6}
A sensor is a couple defined by $(D,f)$, where $D$ is a nonempty closed 
part of $\overline{\Omega}$ representing the geometric support of the sensor, 
and $f$ define the spatial distribution of the information on the support $D$. 
Then the output function \eqref{se12} can be written in the form
\begin{equation}
\label{se13}
z(t) = \displaystyle \int_{D} y(x,t)f(x) dx.
\end{equation}           
\end{definition}

\begin{definition}[See \cite{EL}]
\label{definition2.8}
A sensor may be pointwise (internal or boundary) if $D = \{b\}$ with 
$b \in \overline{\Omega}$ and $f = \delta(b - \cdot)$, where $\delta$ is 
the Dirac mass concentrated in $b$, and the sensor is then denoted by 
$(b,\delta_{b})$. In this case, the operator $C$ is unbounded and 
the output function \eqref{se12} can be written in the form
\begin{equation}
\label{se14}
z(t) = y(b,t).
\end{equation} 
\end{definition}
In order to prove our results, the following lemmas are used.

\begin{lemma}[See \cite{MAG}]
\label{lemma2.8}
Let the reflection operator $\mathcal{Q}$ on the interval $[0, T]$ be defined by
$$
\mathcal{Q} f(t) := f(T-t).
$$
Then the following relations hold:
$$
\mathcal{Q}\, {}^{C}_{0}D_{t}^{q} f(t)
= {}^{C}_{t}D_{T}^{q} \mathcal{Q} f(t), 
\qquad \mathcal{Q} 	{}_{0}I_{t}^{q}   
f(t) = {}_{t}I_{T}^{q} \mathcal{Q} f(t) 
$$
and 
$$
{}^{C}_{0}D_{t}^{q}\mathcal{Q} f(t) 
= \mathcal{Q} \, {}^{C}_{t}D_{T}^{q} f(t), 
\qquad {}_{0}I_{t}^{q} \mathcal{Q} f(t)
= \mathcal{Q}   {}_{t}I_{T}^{q}   f(t).
$$
\end{lemma}

\begin{lemma}[See \cite{SP}]
\label{lemma2.9}
For any $q \in (0,1)$, $t \in [0,T]$, let
$$
H f(t) = \displaystyle \sum_{i=1}^{\infty} \int_{0}^{t} \frac{E_{q,q} 
(\lambda_{i}(t-s)^{q})}{(t-s)^{1-q}}(\varphi_{_{i}},Bf(s))
ds \varphi_{_{i}}(x), \quad f \in L^{2}(0,T;\mathbb{R}^{p}).
$$
Then, 
$$
{}_{0}I_{t}^{1-q} H f(t) = \displaystyle \sum_{i=1}^{\infty} 
\int_{0}^{t} E_{q}(\lambda_{i}(t-s)^{q})(\varphi_{_{i}},Bf(s))ds\varphi_{_{i}}(x),
$$
where $(\varphi_{_{i}})_{i\in \mathbb{N}}$ are the eigenfunctions 
of the operator $A$ in $L^{2}(\Omega)$.
\end{lemma}


\section{Enlarged observability and characterization}
\label{sec:3}

Let $\omega \subseteq \Omega$ be a given subregion with a positive Lebesgue measure. 
We define the restriction operator $\chi_{_{\omega}}$ 
and its adjoint $\chi_{_{\omega}}^{*}$ by
\begin{equation*}
\begin{array}{rll}
\chi_{\omega} : L^{2}(\Omega)& \longrightarrow &  L^{2}(\omega)\\
y & \longrightarrow& \chi_{_{\omega}}y = y_{|\omega} \\
\end{array}
\end{equation*}
and
$$
(\chi_{_{\omega}}^{*}y)(x) 
= 
\begin{cases}
y(x) & \text{ if } \ x \in \omega\\
0 & \text{ if } \ x \in \Omega\backslash\omega.
\end{cases}
$$
From \cite{CZW,DR,PW}, we get that a necessary 
and sufficient condition for the regional 
exact observability of system  \eqref{se11} 
augmented with \eqref{se12} in $\omega$ at time $t$ 
is given by $Im(\chi_{_{\omega}}K_{q}^{*}) = L^{2}(\omega)$.
                                                              
Let $\alpha(\cdot)$ and $\beta(\cdot)$ be two functions defined 
in $L^{2}(\omega)$ such that $\alpha(\cdot) \leq \beta(\cdot)$ 
a.e. in $\omega$. Throughout the paper, we set
\begin{equation*}
\mathcal{E} := \left\{ y\in L^{2}(\omega) \; | \; 
\alpha(\cdot) \leq y(\cdot) \leq \beta(\cdot) \quad \hbox{a.e.}\; 
\hbox{in} \; \omega \; \right\}.
\end{equation*}
We consider
\begin{equation*}
y_{0} = 
\begin{cases}
y_{0}^{1}  & \hbox{in } \; \mathcal{E} \\
y_{0}^{2} & \hbox{in } \; L^{2}(\Omega) \backslash \mathcal{E}.    
\end{cases}
\end{equation*}
The study of regional enlarged observability for Caputo time fractional 
order diffusion systems amounts to solve the following problem.

\begin{problem}
Given the system \eqref{se11} together with the output \eqref{se12} in $\omega$ 
at time $t \in [0,T] $, is it possible to reconstruct $y_{0}^{1}$ between 
two prescribed functions $\alpha(\cdot)$ and $\beta(\cdot)$ in $\omega$?
\end{problem}

Before proving our first result, we need two important definitions.

\begin{definition}
\label{definition3.10}
The system \eqref{se11} together with the output 
\eqref{se12} is said to be exactly 
$\mathcal{E}$-observable in $\omega$ if 
\begin{equation*}
Im (\chi_{_{\omega}} K_{q}^{*}) 
\cap \mathcal{E} \neq \emptyset.
\end{equation*}
\end{definition}

\begin{definition}
\label{definition3.11}
The sensor $(D,f)$ is said to be exactly $\mathcal{E}$-strategic 
in $\omega$ if the observed system is exactly 
$\mathcal{E}$-observable in $\omega$.
\end{definition}

\begin{remark}
If the system \eqref{se11} together with the output \eqref{se12} is exactly 
$\mathcal{E}$-observable in $\omega_{1}$, then it is 
exactly $\mathcal{E}$-observable in any subregion 
$\omega_{2} \subset \omega_{1}$.
\end{remark}

\begin{theorem}
\label{Proposition3.13}
The following two statements are equivalent:
\begin{enumerate}
\item[\textbf{1.}] The system \eqref{se11} together 
with the output \eqref{se12} is exactly 
$\mathcal{E}$-observable in $\omega$.

\item[\textbf{2.}] $Ker (K_{q} \chi_{_{\omega}}^{*}) 
\cap \mathcal{E} = \{0\}$.
\end{enumerate}
\end{theorem}

\begin{proof}
We begin by proving that statement 1 implies 2.
For that we show that
$$
\begin{array}{rll}
Im (\chi_{_{\omega}} K_{q}^{*}) 
\cap \mathcal{E} \neq \emptyset
&\Longrightarrow
& Ker (K_{q} \chi_{_{\omega}}^{*}) 
\cap \mathcal{E}  = \{0\}.
\end{array}
$$
Suppose that
$$
Ker (K_{q} \chi_{_{\omega}}^{*}) 
\cap \mathcal{E} \neq \{0\}.
$$
Let us consider $y \in Ker (K_{q} \chi_{_{\omega}}^{*}) \cap\mathcal{E}$ 
such that $y \neq 0$. Then, $y \in Ker (K_{q} \chi_{_{\omega}}^{*}) $ and 
$y \in \mathcal{E}$.
We have $Ker (K_{q} \chi_{_{\omega}}^{*}) = Im(\chi_{_{\omega}} K_{q}^{*})^{\perp}$, 
so that $y \in Im(\chi_{_{\omega}} K_{q}^{*})^{\perp}$, $y \neq 0$. 
Therefore, $y \notin Im(\chi_{_{\omega}} K_{q}^{*})$, and 		
\begin{gather*}
Ker (K_{q} \chi_{_{\omega}}^{*}) \cap\mathcal{E} 
\subset \displaystyle L^{2}(\omega)\setminus Im(\chi_{_{\omega}} K_{q}^{*}),\\
Im(\chi_{_{\omega}} K_{q}^{*})\subset \displaystyle 
\left[ L^{2}(\omega)\setminus Ker (K_{q} \chi_{_{\omega}}^{*})\right]  
\cup \displaystyle \left[ L^{2}(\omega)\setminus \mathcal{E} \,\right].
\end{gather*}
We have 
$$
Im(\chi_{_{\omega}} K_{q}^{*}) \subset \displaystyle L^{2}(\omega) 
\setminus Ker (K_{q} \chi_{_{\omega}}^{*}).
$$
Accordingly,
$$
Im(\chi_{_{\omega}} K_{q}^{*}) 
\cap Ker (K_{q} \chi_{_{\omega}}^{*}) = \emptyset
$$
and
$$
Im(\chi_{_{\omega}} K_{q}^{*}) \cap Im(\chi_{_{\omega}} 
K_{q}^{*})^{\perp} = \emptyset,
$$ 
which is absurd. Since
$$
Im(\chi_{_{\omega}} K_{q}^{*}) \subset \displaystyle 
\displaystyle  L^{2}(\omega)\setminus \mathcal{E},
$$
it follows that
$$
Im(\chi_{_{\omega}} K_{q}^{*}) \cap \mathcal{E}  
= \emptyset,
$$ 
which is also absurd. Consequently, 
$$
Ker (K_{q} \chi_{_{\omega}}^{*}) \cap \mathcal{E} = \{0\}.
$$
We now prove the reverse implication: statement 2 implies 1. 
For that we show that
$$
\begin{array}{rll}
Ker (K_{q} \chi_{_{\omega}}^{*}) 
\cap \mathcal{E} = \{0\}
&\Longrightarrow& Im(\chi_{_{\omega}} K_{q}^{*}) 
\cap \mathcal{E} \neq \emptyset.
\end{array}
$$
Suppose that
$$
Ker (K_{q} \chi_{_{\omega}}^{*}) 
\cap \mathcal{E} = \{0\}.
$$
Let us consider 
$$
y \in Ker (K_{q} \chi_{_{\omega}}^{*})  \cap\mathcal{E}.
$$
Then, $y \in Ker (K_{q} \chi_{_{\omega}}^{*})$ and 
$y \in \mathcal{E}$ such that $y = 0$.
We have 
$$
Ker (K_{q} \chi_{_{\omega}}^{*}) = Im(\chi_{_{\omega}} 
K_{q}^{*})^{\perp}, 
$$
so $y \in Im(\chi_{_{\omega}} K_{q}^{*})^{\perp}$ such that $y = 0$. Hence, 
$$
y \in Im(\chi_{_{\omega}} K_{q}^{*})\; \hbox{and}\; y \in \mathcal{E},
$$ 
and
$$
Im(\chi_{_{\omega}} K_{q}^{*}) \cap \mathcal{E}  \neq \emptyset,
$$
which shows that \eqref{se11}--\eqref{se12} is exactly 
$\mathcal{E}$-observable in $\omega$.
\end{proof}

\begin{remark}
There exist systems that are not observable in the whole domain but exactly 
$\mathcal{E}$-observable in some region. This is illustrated 
by the following example (see Proposition~\ref{prop:15}).
\end{remark}

\begin{example}
\label{example1}
Let us consider the following one-dimensional time fractional 
differential system of order $q \in (0,1)$ in 
$\Omega_{1} = [0,1]$, excited by a pointwise sensor:
\begin{equation}
\label{se15}
\begin{cases}
^{C}_{0}D_{t}^{0.6}y(x,t)
= \displaystyle \frac{\partial^{2}} {\partial x^{2}} y(x,t)
& \hbox{in} \quad  [0,1]\times[0,T] \\
y(0,t) \;= \;y(1,t) = 0 &  \hbox{in} \quad [0,T]\\ 
y(x,0)  = y_{0}(x) &  \hbox{in} \quad [0,1],
\end{cases}
\end{equation}
augmented with the output function
\begin{equation}
\label{se16}
z(t) = C y(x,t) = y(b,t),
\end{equation}
where $b= \displaystyle \frac{1}{3} \in \Omega_{1}$.
The operator $A = \displaystyle \frac{\partial^{2}}{\partial x^{2}}$ 
has a complete set of eigenfunctions $(\varphi_{_{i}})$ in $L^{2}(\Omega_{1})$ 
associated with the eigenvalues $(\lambda_{i})$, given by
$$
\varphi_{_{i}}(x) = \sqrt{2} \sin(i \pi x) 
\quad \hbox{and} \quad \lambda_{i} = -i^{2}\pi^{2}
$$
with
$$
R(t)y(x) = \sum_{i=1}^{\infty} e^{\lambda_{i} t} \left\langle y, 
\varphi_{_{i}} \right\rangle_{_{L^{2}(\Omega_{1})}} \varphi_{_{i}}(x).
$$
Then,
$$
S_{0.6}(t)y(x) = \sum_{i=1}^{\infty} 
E_{0.6} (\lambda_{i} t^{0.6})\left\langle y, 
\varphi_{_{i}} \right\rangle_{_{L^{2}(\Omega_{1})}} \varphi_{_{i}}(x),
$$
where $E_{q}(z) := \displaystyle 
\sum_{i=0}^{\infty}\frac{z^{i}}{\Gamma(q i + 1)}$, 
\textbf{Re} $q>0$, $ z \in \mathbb{C}$, 
is the generalized  Mittag-Leffler function in one parameter (see, e.g., \cite{GO}).

Let $y_{0}(x) = \sin (2\pi x)$ be the initial state to be observed. Then, 
for $\omega_{1} = \displaystyle \left[\frac{1}{4}, \frac{1}{2}\right]$, 
the following result holds.

\begin{proposition}
\label{prop:15}
There is a state for which the system \eqref{se15}--\eqref{se16} 
is not weakly observable in $\Omega_{1}$ but it is exactly 
$\mathcal{E}_{1}$-observable in $\omega_{1}$.
\end{proposition}

\begin{proof}
To show that system \eqref{se15}--\eqref{se16} is not weakly 
observable in $\Omega_{1}$, it is sufficient to verify that 
$y_{0} \in Ker(K_{0.6})$. We have
$$
\begin{array}{rll}
K_{0.6}\, y_{0}(x)  
& = & \displaystyle \sum_{i=1}^{\infty} 
E_{0.6} (\lambda_{i} t^{0.6})\left\langle y_{0}, 
\varphi_{_{i}} \right\rangle_{_{L^{2}(\Omega_{1})}} \varphi_{_{i}}(b)\\
& = & \displaystyle 2 \sum_{i=1}^{\infty} 
E_{0.6} (\lambda_{i} t^{0.6}) \sin\left(\frac{i \pi}{3}\right) 
\int_{0}^{1} \sin(2 \pi x) \sin(i \pi x) dx \\
& = & 0.
\end{array}
$$
Hence, $K_{0.6} y_{0}(x) = 0$. Consequently, the state $y_{0}$ 
is not weakly observable in $\Omega_{1}$.
On the other hand, one has
$$
\begin{array}{rll}
K_{0.6} \chi_{_{\omega_{1}}}^{*} \chi_{_{\omega_{1}}} y_{0}(x)  
& = & \displaystyle \sum_{i=1}^{\infty}
E_{0.6} (\lambda_{i} t^{0.6})\left\langle 
\chi_{_{\omega_{1}}}^{*} \chi_{_{\omega_{1}}}y_{0}, 
\varphi_{_{i}} \right\rangle_{_{L^{2}(\Omega_{1})}} \varphi_{_{i}}(b)\\
& = & \displaystyle \sum_{i=1}^{\infty} 
E_{0.6} (\lambda_{i} t^{0.6})\left\langle y_{0}, 
\varphi_{_{i}} \right\rangle_{_{L^{2}(\omega_{1})}} \varphi_{_{i}}(b)\\
& = & \displaystyle \sqrt{3} E_{0.6} (- \pi^{2} t^{0.6}) 
\int_{\frac{1}{4}}^{\frac{1}{2}} \sin(2 \pi x) \sin(\pi x) dx \\
& = & \displaystyle \frac{4\sqrt{3} - \sqrt{6}}{6 \pi} 
E_{0.6} (- \pi^{2} t^{0.6})\\
& \neq & 0,
\end{array}
$$
which means that the state $y_{0}$ is weakly observable in $\omega_{1}$.
Moreover, for 
$$
\alpha_{1} (x) = \left| y_{|\omega_{1}}^{0}(x) \right| - \frac{1}{2} 
< y_{|\omega_{1}}^{0}(x) 
$$ 
and 
$$
\beta_{1} (x) 
= \left| y_{|\omega_{1}}^{0}(x) \right| + \frac{1}{2} 
> y_{|\omega_{1}}^{0}(x), \quad \forall x \in \omega_{1},
$$ 
we have $\chi_{_{\omega_{1}}} y_{0}(x) \in \mathcal{E}_{1}$ 
and system \eqref{se15}--\eqref{se16} is exactly 
$\mathcal{E}_{1}$-observable in $\omega_{1}$.\\
The proof is complete.
\end{proof}
\end{example}


\section{HUM approach}
\label{sec:4}

Here we will use an extension of the Hilbert Uniqueness Method (HUM) 
introduced by Lions (see \cite{LI}) to reconstruct the initial state 
between two prescribed functions $\alpha(\cdot)$ and $\beta(\cdot)$ 
in $\omega$. In what follows, $\mathcal{G}$ is defined by
\begin{equation}
\label{se17}
\mathcal{G} = \left\{\, g \in L^{2}(\Omega) \; | \; g = 0 \quad 
\hbox{in} \; L^{2}(\Omega) \backslash \mathcal{E} \, \right\}.
\end{equation}
For $\varphi_{0} \in \mathcal{G}$, we consider the following system:
\begin{equation}
\label{se18}
\begin{cases}
^{C}_{0}D_{t}^{q}\varphi(x,t)
=  A \varphi(x,t) & \hbox{in}  \quad  Q_{T} \\
\varphi(\xi,t) = 0 &  \hbox{on} \quad \Sigma_{T}\\ 
\varphi(x,0) = \varphi_{_{0}}(x) &  \hbox{in} \quad \Omega.
\end{cases}
\end{equation}
Without loss of generality, we denote $\varphi(x, t) := \varphi(t)$. 
System \eqref{se18} admits a unique mild solution 
$\varphi \in L^{2}(0,T;L^{2}(\Omega))$ 
given by $\varphi(t) =  S_{q}(t)\varphi_{_{0}}$. Now we go
further into the state reconstruction by considering two types of sensors.


\subsection{Pointwise sensors}

In this case, the output function is given by
\begin{equation}
\label{se19}
z(t) = \varphi(b,T-t), \quad t \in [0,T],
\end{equation}
where $b\in \Omega$ denotes the given location of the sensor. 
We consider a semi-norm on $\mathcal{G}$ defined by
\begin{equation}
\label{se110}
\varphi_{_{0}} \longmapsto \| \varphi_{_{0}} \|_{\mathcal{G}}^{2} 
= \displaystyle  \int_{0}^{T} \left\| C \varphi(T-t) \right\|^{2}  dt.
\end{equation}
The following result holds.

\begin{lemma}
\label{lemma4.16}
If the system \eqref{se11} together with the output \eqref{se19} 
is exactly $\mathcal{E}$-observable in $\omega$, 
then \eqref{se110} defines a norm on $\mathcal{G}$.
\end{lemma}

\begin{proof}
Consider $\varphi_{_{0}} \in \mathcal{G}$. Then,
$$
\left\| \varphi_{_{0}} \right\|_\mathcal{G} 
\quad \Longrightarrow \quad C\varphi(T-t) = 0 
\quad \hbox{for all}  \quad t \in [0,T].
$$ 
We have 
$$
\varphi_{_{0}} \in L^{2}(\Omega) \quad \Longrightarrow 
\quad \chi_{_{\omega}}\varphi_{_{0}} \in L^{2}(\omega)
$$
or 
$$
K_{q}(t) \chi_{_{\omega}}^{*}\chi_{_{\omega}}\varphi_{_{0}} 
= CS_{q}(t) \chi_{_{\omega}}^{*}\chi_{_{\omega}}\varphi_{_{0}} = 0.
$$
Hence, 
$$
\chi_{_{\omega}}\varphi_{_{0}} \in Ker(K_{q}\chi_{\omega}^{*}).
$$
For $\chi_{_{\omega}} \varphi_{_{0}} \in \mathcal{E}$, 
one has $\chi_{_{\omega}} \varphi_{_{0}} \in Ker(K_{q} \chi_{\omega}^{*}) 
\cap\, \mathcal{E} $ and, because the system is exactly 
$\mathcal{E}$-observable in $\omega$, $\chi_{_{\omega}}\varphi_{_{0}} = 0$. 
Consequently, $\varphi_{_{0}} = 0$ and \eqref{se110} is a norm.
The proof is complete.
\end{proof}

Consider the system 
\begin{equation}
\label{se111}
\begin{cases}
\mathcal{Q}\,^{C}_{t}D_{T}^{q}\Psi(x,t)
=  A^{*} \mathcal{Q}\Psi(x,t) + C^{*}C  \mathcal{Q} \varphi(x,t)
& \hbox{in} \quad  Q_{T} \\
\Psi(\xi,t) = 0  & \hbox{on} \quad \Sigma_{T}\\ 
\Psi(x,T) 
= 0 &  \hbox{in} \quad \Omega.
\end{cases}
\end{equation}
For $\varphi_{_{0}} \in \mathcal{G}$, 
we define the operator $\Lambda : \mathcal{G} \longrightarrow \mathcal{G}^{*}$ by
\begin{equation*}
\mathcal{N} \varphi_{_{0}} =  \mathcal{P}({}_{0}I_{T}^{1-q}\Psi(0)),
\end{equation*}  
where $\mathcal{P} = \chi_{\omega}^{*}\chi_{\omega}$ and $\Psi(0) = \Psi(x,0)$.
Let us now consider the system 
\begin{equation}
\label{se112}
\begin{cases}
\mathcal{Q}\, ^{C}_{t}D_{T}^{q}\Theta(x,t)
=  A^{*} \mathcal{Q} \Theta(x,t) + C^{*} \mathcal{Q}z(t)
& \hbox{in} \quad  Q_{T} \\
\Theta(\xi,t) = 0 &  \hbox{on} \quad \Sigma_{T}\\ 
 \Theta(x,T) = 0 &  \hbox{in} \quad \Omega.
\end{cases}
\end{equation}
If $\varphi_{_{0}}$ is chosen such that $\Theta(0) = \Psi(0)$ in $\omega$, 
then our problem of enlarged observability is reduced to solve the equation
\begin{equation}
\label{se113}
\mathcal{N} \varphi_{_{0}} = \mathcal{P}({}_{0}I_{T}^{1-q}\Theta(0)).
\end{equation} 

\begin{theorem}
\label{theorem4.17}
If the system \eqref{se11} together with the output \eqref{se19} is exactly 
$\mathcal{E}$-observable in $\omega$, then equation 
\eqref{se113} admits a unique solution $\varphi_{_{0}} \in \mathcal{G}$, 
which coincides with the initial state $y_{0}^{1}$ observed between  
$\alpha(\cdot)$ and $\beta(\cdot)$ in $\omega$. Moreover,
$y_{0}^{1} = \chi_{_{\omega}} \varphi_{_{0}}$.
\end{theorem}

\begin{proof}
By Lemma~\ref{lemma4.16}, if the system \eqref{se11} 
together with the output \eqref{se19} is exactly 
$\mathcal{E}$-observable in $\omega$, we see that $\left\| \cdot \right\|_{\mathcal{G}}$ 
is a norm of the space $\mathcal{G}$.\\
Now, we show that \eqref{se113} admits a unique solution in $\mathcal{G}$. 
For any $\varphi_{_{0}} \in \mathcal{G}$, equation \eqref{se113} admits 
a unique solution if $\mathcal{N}$ is an isomorphism. Then, 
$$
\begin{array}{rll}
\left\langle \mathcal{N} \varphi_{_{0}}, \varphi_{_{0}} \right\rangle_{L^{2}(\Omega)}  
& = & \left\langle \mathcal{P}({}_{0}I_{T}^{1-q}\Psi(0)), 
\varphi_{_{0}} \right\rangle_{L^{2}(\Omega)} \\
& = & \left\langle \chi_{_{\omega}}^{*}\chi_{_{\omega}} 
({}_{0}I_{T}^{1-q}\Psi(0)), \varphi_{_{0}} \right\rangle_{L^{2}(\Omega)}  \\
& = & \left\langle {}_{0}I_{T}^{1-q}\Psi(0), \varphi_{_{0}} \right\rangle_{L^{2}(\omega)}.
\end{array}
$$
Moreover, by Lemma~\ref{lemma2.8}, we see that system 
\eqref{se111} can be rewritten as
\begin{equation*}
\begin{cases}
^{C}_{0}D_{t}^{q} \mathcal{Q}\Psi(x,t)
=  A^{*} \mathcal{Q}\Psi(x,t) + C^{*}C  \mathcal{Q} \varphi(x,t)
& \hbox{in} \quad  Q_{T} \\
\Psi(\xi,t) = 0  & \hbox{on} \quad \Sigma_{T}\\ 
\Psi(x,T) 
= 0 &  \hbox{in} \quad \Omega,
\end{cases}
\end{equation*}
and its unique mild solution is
$$
\Psi(t) = S^{*}_{q}(T-t)\Psi(T) 
+ \int_{t}^{T}(T-\tau)^{q-1} H^{*}_{q}(T-\tau)C^{*}C\varphi(T-\tau) d\tau
$$
and
$$
\Psi(0) = \displaystyle \int_{0}^{T} (T-\tau)^{q-1} H^{*}_{q}(T-\tau) 
C^{*}C\varphi(T-\tau) d\tau,
$$
where
$$
H^{*}_{q}(t) = q  \int_{0}^{\infty} 
\theta \xi_{q} (\theta) R^{*}(t^{q}\theta) d\theta
$$
with $(R^{*}(t))_{t \geq 0}$ the strongly continuous semi-group generated by $A^{*}$. 
We obtain  by Lemma~\ref{lemma2.9} that
$$
\begin{array}{rll}
\left\langle \mathcal{N} \varphi_{_{0}}, \varphi_{_{0}} \right\rangle_{L^{2}(\Omega)}  
& = & \left\langle {}_{0}I_{T}^{1-q}\Psi(0), \varphi_{_{0}} \right\rangle\\
& = & \displaystyle \left\langle  \int_{0}^{T} S^{*}_{q}(T-\tau) 
C^{*}C\varphi(T-\tau) d\tau, \varphi_{_{0}} \right\rangle\\
& = & \displaystyle  \int_{0}^{T} \displaystyle \left\langle 
C\varphi(T-\tau), C S_{q}(T-\tau) \varphi_{_{0}} \right\rangle d\tau\\
& = & \displaystyle \int_{0}^{T} \displaystyle \left\| 
C \varphi(T-\tau) \right\|^{2} d \tau\\
& = & \displaystyle \left\| \varphi_{_{0}} \right\|^{2}_{\mathcal{G}}.
\end{array}
$$
Then the equation \eqref{se113} has a unique solution that is also the initial 
state to be estimated between $\alpha(\cdot)$ and 
$\beta(\cdot)$ in the subregion $\omega$ given by
$$
y_{0}^{1} = \chi_{_{\omega}} \varphi_{_{0}}.
$$
The proof is complete.
\end{proof}


\subsection{Zone sensors}

Let us come back to system \eqref{se11} and suppose that the measurements 
are given by an internal zone sensor defined by $(D,f)$. The system 
is augmented with the output function
\begin{equation}
\label{se114}
z(t) = \displaystyle \int_{D} y(x,T-t) f(x) dx .
\end{equation}
In this case, we consider \eqref{se18}, $\mathcal{G}$ 
given by \eqref{se17}, and we define a semi-norm on $\mathcal{G}$ by
\begin{equation}
\label{se115}
\| \varphi_{_{0}}\|_{\mathcal{G}}^{2} 
= \displaystyle \int_{0}^{T} \displaystyle 
\left\langle  \varphi(T-t), f \right\rangle _{L^{2}(D)}^{2}dt
\end{equation}
with
\begin{equation}
\label{se116}
\begin{cases}
\mathcal{Q} \; ^{C}_{t}D_{T}^{q}\Psi(x,t)
=  A^{*} \mathcal{Q} \Psi(x,t) + \left\langle 
\mathcal{Q}\varphi(t), f \right\rangle_{L^{2}(D)} 
\chi_{_{D}}f(x)  & \hbox{in} \quad  Q_{T} \\
\Psi(\xi,t) = 0 &  \hbox{on} \quad \Sigma_{T}\\ 
\Psi(x,T) = 0 
& \hbox{in} \quad \Omega.
\end{cases}
\end{equation}
We introduce the operator
\begin{equation}
\label{se117}
\begin{array}{rll}
\mathcal{N} : \mathcal{G}& \longrightarrow &  \mathcal{G}^{*}\\
\varphi_{_{0}} & \longrightarrow
& \mathcal{N} \varphi_{_{0}} = \mathcal{P}({}_{0}I_{T}^{1-q}\Psi(0)),
\end{array}
\end{equation}   
where $\mathcal{P} = \chi_{\omega}^{*}\chi_{\omega}$ and $\Psi(0) = \Psi(x,0)$.
 Let us consider the system 
\begin{equation}
\label{se118}
\begin{cases}
\mathcal{Q} \; ^{C}_{t}D_{T}^{q}\Theta(x,t)
=  A^{*} \mathcal{Q} \Theta(x,t) + \left\langle \mathcal{Q} z(t), 
f \right\rangle _{L^{2}(D)} \chi_{_{D}}f(x) 
& \hbox{in} \quad  Q_{T} \\
\Theta(\xi,t) = 0 &  \hbox{on} \quad \Sigma_{T}\\ 
\Theta(x,T) 
= 0 &  \hbox{in} \quad \Omega.
\end{cases}
\end{equation}
If $\varphi_{_{0}}$ is chosen such that $\Theta(0) = \Psi(0)$ in $\omega$, 
then \eqref{se118} can be seen as the adjoint of system \eqref{se11} 
and our problem of enlarged observability consists to solve the equation
\begin{equation}
\label{se119}
\mathcal{N} \varphi_{_{0}} = \mathcal{P}({}_{0}I_{T}^{1-q} \Theta(0)).
\end{equation} 

\begin{theorem}
If system \eqref{se11} together with the output \eqref{se114} 
is exactly $\mathcal{E}$-observable in $\omega$, 
then equation \eqref{se119} has a unique solution 
$\varphi_{_{0}} \in \mathcal{G}$, which coincides with the initial 
state $y_{0}^{1}$ observed between $\alpha(\cdot)$ 
and $\beta(\cdot)$ in $\omega$.
\end{theorem}

\begin{proof}
The proof is similar to the proof of Theorem~\ref{theorem4.17}.
\end{proof}


\section{Example}
\label{sec:5}

Let us consider the system \eqref{se15} together with the output 
\eqref{se16} and let $b \in \Omega_{1}= [0,1] $,\\ 
$\displaystyle\omega_{1}=\left[ \frac{1}{4},\frac{1}{2}\right] $, 
$\alpha_{1} (x) = \displaystyle \left| y_{|\omega_{1}}^{0}(x) \right| - \frac{1}{2}$, 
$\beta_{1} (x) = \displaystyle \left| y_{|\omega_{1}}^{0}(x) \right| + \frac{1}{2}$ 
and   $\mathcal{G}_{1}$ be the set defined by
$$
\mathcal{G}_{1} = \left\{\, g \in L^{2}(\Omega_{1}) \; | \; g = 0 
\quad \hbox{in} \; L^{2}(\Omega_{1}) \backslash 
\mathcal{E}_{1} \, \right\}.
$$
By Lemma~\ref{lemma4.16}, if the system \eqref{se15} together with the output 
\eqref{se16} is exactly $\mathcal{E}_{1}$-observable in $\omega_{1}$, then, 
for any $\varphi_{_{0}}\in \mathcal{G}_{1}$, we see that
$$
\varphi_{_{0}} \longmapsto \| \varphi_{_{0}} \|_{\mathcal{G}_{1}}^{2} 
= \displaystyle  \int_{0}^{T} \left\| 
\varphi(b,T-t) \right\|^{2}  dt
$$
defines a norm on $\mathcal{G}_{1}$, where $\varphi(x,t)$ solves
\begin{equation*}
\begin{cases}
^{C}_{0}D_{t}^{0.6}\varphi(x,t)
=  \displaystyle \frac{\partial^{2}} {\partial x^{2}} \varphi(x,t) 
& \hbox{in}  \quad  [0,1]\times[0,T] \\
\varphi(0,t) = \varphi(1,t) = 0 &  \hbox{in} \quad [0,T]\\ 
\varphi(x,0) = \varphi_{_{0}}(x) &  \hbox{in} \quad [0,1].
\end{cases}
\end{equation*}
The regional enlarged observability problem is equivalent 
to solving the equation 
\begin{equation}
\label{se120}
\mathcal{N}_{1} \varphi_{_{0}} 
= \displaystyle\mathcal{P} ({}_{0}I_{T}^{0.4}\,\Theta(0)),
\end{equation}
where $\Theta$ satisfies
$$
\begin{cases}
\mathcal{Q} \; ^{C}_{t}D_{T}^{0.6}\Theta(x,t)
=  \displaystyle \frac{\partial^{2}} {\partial x^{2}} 
\mathcal{Q}  \Theta(x,t) + \delta(x-b)z(T-t)
& \hbox{in} \quad  [0,1] \times [0,T] \\
\Theta(0,t) = \Theta(1,t) = 0 & \hbox{in} \quad  [0,T] \\ 
\Theta(x,T) 
= 0 &  \hbox{in} \quad [0,1].
\end{cases}
$$
Let $\mathcal{N}_{1}$ be defined by
$$
\mathcal{N}_{1} \varphi_{_{0}} 
= \displaystyle\mathcal{P} ({}_{0}I_{T}^{0.4}\,\Psi(0)),
$$
which is an isomorphism from $\mathcal{G}_{1}$ to $\mathcal{G}_{1}^{*}$, 
and $\Psi$ be the solution of the following system:
\begin{equation*}
\begin{cases}
\mathcal{Q}\,^{C}_{t}D_{T}^{0.6}\Psi(x,t)
=  \displaystyle \frac{\partial^{2}} {\partial x^{2}}  
\mathcal{Q}\Psi(x,t) + C^{*}C  \mathcal{Q} \varphi(x,t)
& \hbox{in} \quad  [0,1]\times[0,T] \\
\Psi(0,t) = \Psi(1,t) = 0  & \hbox{in} \quad [0,T]\\ 
\Psi(x,T) 
= 0 &  \hbox{in} \quad [0,1].
\end{cases}
\end{equation*}
By Theorem~\ref{theorem4.17}, we conclude that \eqref{se120} admits 
a unique solution $\varphi_{_{0}} \in \mathcal{G}_{1}$ and 
$y_{|\omega_{1}}^{0} = \chi_{_{\omega_{1}}}\varphi_{_{0}}$, 
provided that the system \eqref{se15}--\eqref{se16} is exactly 
$\mathcal{E}_{1}$-observable in $\omega_{1}$.


\section{Conclusion}
\label{sec:6}

In this paper, regional enlarged observability of Caputo time 
fractional diffusion systems of order $\alpha \in (0,1)$ is discussed. 
The results we present here can also be extended to complex fractional 
order distributed parameter systems. This and other questions, as to give 
numerical results and a real application to support our theoretical analysis, 
are being considered and will be addressed elsewhere.


\section*{Acknowledgments}

This work was supported by Hassan II Academy of Sciences and Technology 
project 630/2016 and by Portuguese funds through the Center for Research 
and Development in Mathematics and Applications (CIDMA) and The Portuguese 
Foundation for Science and Technology (FCT), within project UID/MAT/04106/2013.



\medskip

Received April 2018; revised July 2018.

\medskip


\end{document}